\documentclass[12pt, a4paper]{amsart}
\usepackage{a4, a4wide}
\usepackage{amssymb}
\usepackage{amsmath}
\usepackage{amsthm}
\usepackage{amstext}
\usepackage{amscd}
\usepackage{latexsym}
\usepackage{graphics}
\usepackage{color}
\usepackage{enumitem}
\usepackage[all]{xy}

\usepackage[colorlinks, pagebackref]{hyperref}
\hypersetup{
  colorlinks=true,
  citecolor=blue,
  linkcolor=blue,
  urlcolor=blue}
\makeatletter
\def\@tocline#1#2#3#4#5#6#7{\relax
  \ifnum #1>\c@tocdepth % then omit
  \else
    \par \addpenalty\@secpenalty\addvspace{#2}%
    \begingroup \hyphenpenalty\@M
    \@ifempty{#4}{%
      \@tempdima\csname r@tocindent\number#1\endcsname\relax
    }{%
      \@tempdima#4\relax
    }%
    \parindent\z@ \leftskip#3\relax \advance\leftskip\@tempdima\relax
    \rightskip\@pnumwidth plus4em \parfillskip-\@pnumwidth
    #5\leavevmode\hskip-\@tempdima
      \ifcase #1
      \or\or \hskip 2em \or \hskip 2homologyem \else \hskip 3em \fi%
      #6\nobreak\relax
    \dotfill\hbox to\@pnumwidth{\@tocpagenum{#7}}\par
    \nobreak
    \endgroup
  \fi}
\makeatother

\theoremstyle{plain}

\newtheorem{introtheorem}{Theorem}[]

\newtheorem{theorem}{Theorem}[section]

\newtheorem{lemma}[theorem]{Lemma}
\newtheorem{corollary}[theorem]{Corollary}
\newtheorem{proposition}[theorem]{Proposition}

\theoremstyle{definition}

\newtheorem{notation}[theorem]{Notation}
\newtheorem{remark}[theorem]{Remark}

\newtheorem{definition}[theorem]{Definition}

\numberwithin{equation}{section}
\newcommand{\Union}{\bigcup}

\newcommand{\ord}{{\rm ord}}
\newcommand{\mult}{{\rm mult}}
\newcommand{\CH}{{\rm CH}}

\newcommand{\Spec}{{\rm Spec \,}}

\newcommand{\Sym}{{\rm Sym}}

\newcommand{\Gal}{{\rm Gal}}

\renewcommand{\tilde}{\widetilde}
\newcommand{\N}{{\mathbb N}}

\newcommand{\Z}{{\mathbb Z}}
\newcommand{\A}{{\mathbb A}}
\renewcommand{\P}{{\mathbb P}}

\newcommand{\KM}{{\rm K}^{\rm M}} 		% Milnor K-theory
\def\<{\langle}
\def\>{\rangle} 
\def\-{\overline} 
\def\~{\widetilde}
\def\^{\widehat}

\def\@{\mathcal}
\def\!{\mathscr}
\def\#{\mathbb}
\def\_{\underline}

\input{xy}
\xyoption{all}
\begin{document}

\title{Geometric criteria for $\A^1$-connectedness and applications to norm varieties}

\author{Chetan Balwe}
\address{Department of Mathematical Sciences, Indian Institute of Science Education and Research Mohali, Knowledge City, Sector-81, Mohali 140306, India.}
\email{cbalwe@iisermohali.ac.in}

\author{Amit Hogadi}
\address{Department of Mathematical Sciences, Indian Institute of Science Education and Research Pune, Dr. Homi Bhabha Road, Pashan, Pune 411008, India.}
\email{amit@iiserpune.ac.in}

\author{Anand Sawant}
\address{School of Mathematics, Tata Institute of Fundamental Research, Homi Bhabha Road, Colaba, Mumbai 400005, India.}
\email{asawant@math.tifr.res.in}
\date{}
\thanks{Chetan Balwe acknowledges the support of SERB-DST MATRICS Grant: MTR/2017/000690}
\thanks{Anand Sawant acknowledges the support of SERB Start-up Research Grant SRG/2020/000237 and the Department of Atomic Energy, Government of India, under project no. 12-R\&D-TFR-5.01-0500.}
%\subjclass[2010]{14C15, 14C25, 19E15 (Primary)}
%\keywords{algebraic cycles; motivic cohomology; rost nilpotence}

\begin{abstract} 
We show that $\A^1$-connectedness of a large class of varieties over a field $k$ can be characterized as the condition that their generic point can be connected to a $k$-rational point using (not necessarily naive) $\A^1$-homotopies.  We also show that symmetric powers of $\A^1$-connected smooth projective varieties (over an arbitrary field) as well as smooth proper models of them (over an algebraically closed field of characteristic $0$) are $\A^1$-connected.  As an application of these results, we show that the standard norm varieties over a  field $k$ of characteristic $0$ become $\A^1$-connected (and consequently, universally $R$-trivial) after base change to an algebraic closure of $k$. 
\end{abstract}

\maketitle
%\tableofcontents

\setlength{\parskip}{2pt plus1pt minus1pt}

\section{Introduction}
\label{section introduction}

The study of near-rationality properties of algebraic varieties is a central problem in algebraic geometry.  This paper is motivated by the following question - do the standard norm varieties over a field constructed by Rost-Voevodsky \cite{Rost-ICM2002} (see also \cite{Suslin-Joukhovitski}) admit any near-rationality properties, after passing to the algebraic closure of the base field?  The standard norm varieties are a crucial input in the proof of the Bloch-Kato conjecture \cite{Voevodsky-Bloch-Kato}.  It has been shown in \cite[Proposition 2.6 and Remark 2.9]{Asok-unramified} that standard norm varieties over a field $k$ (corresponding to a symbol in the mod-$\ell$ Milnor $K$-group of $k$) are rationally connected after base change to $\-k$.  However, it is not known if they are retract rational.  Over a field $k$ of characteristic $0$, various near-rationality properties of varieties are related as follows:
\[
\begin{split}
\text{rational}~ \Rightarrow ~ \text{stably rational} ~ \Rightarrow ~ & \text{retract rational} ~ \Rightarrow ~ \text{unirational} ~ \Rightarrow ~ \text{rationally connected} \\
& \quad \quad \quad \Downarrow\\
& ~ \text{$\A^1$-connected}
\end{split}
\]
where a variety $X$ is $\A^1$-connected if the sheaf of its $\A^1$-connected components $\pi_0^{\A^1}(X)$ (in the sense of \cite{Morel-Voevodsky}) is the trivial (single point) sheaf.  See \cite[Section 1]{CT-Sansuc-rationality}, for example, for the horizontal implications (over an arbitrary field) and \cite[Theorem 2.3.6]{Asok-Morel} for the vertical implication.  It is an open question, whether $\A^1$-connectedness is equivalent to retract rationality \cite[Remark 2.3.11]{Asok-Morel}.  

Pfister quadrics corrresponding to symbols in mod-$2$ Milnor $K$-theory are norm varieties and are obviously $\A^1$-connected over an algebraically closed field, being rational.  By the main results of \cite{Nguyen}, Pfister quadrics are birational to standard norm varieties.  Consequently, standard norm varieties in the case $\ell = 2$ are rational and hence, $\A^1$-connected.  

Let $\ell$ be an arbitrary prime number.  Given an ordered sequence $\alpha = \{a_1, \ldots, a_n\}$ in $k^\times$, the standard norm variety $X_\alpha$ is constructed by an inductive procedure (see Section \ref{section norm varieties} for an outline of the construction following \cite{Rost-ICM2002} and \cite{Suslin-Joukhovitski}).  The standard norm variety $X_\alpha$ has nice splitting properties with respect to the symbol $\{a_1, \ldots, a_n\} \in \KM_n(k)/\ell$ (see \cite{Suslin-Joukhovitski}).  Note that it is not clear how standard norm varieties associated with given two ordered sequences of units $a_1, \ldots, a_n$ and $b_1, \ldots, b_n$ in $k$ such that equality of the associated symbols $\{a_1, \ldots, a_n\} = \{b_1, \ldots, b_n\} \in \KM_n(k)/\ell$ holds are related.  It has been conjectured \cite[Conjecture 1.1]{Nguyen} that such standard norm varieties are birational.  For a pair of units $\{a_1, a_2\}$, the standard norm variety is the Severi-Brauer variety associated with the cyclic algebra corresponding to the symbol $\{a_1, a_2\}  \in \KM_2(k)/\ell$, which is $\A^1$-connected precisely when it admits a $k$-rational point (since it is isomorphic to a projective space).  The main result of this paper shows that a standard norm variety is $\A^1$-connected, after passing to the algebraic closure $\-k$ of the base field $k$ (see Theorem \ref{theorem:norm varieties}).  

\begin{introtheorem}
\label{intro thm: norm varieties}
Let $k$ be a field of characteristic $0$ and let $\ell$ be any prime number.  Let $n\geq 2$ be an integer and $a_1, \ldots, a_n \in k^{\times}$.  Let $X_\alpha$ denote the standard norm variety associated with the ordered sequence $\alpha = \{a_1, \ldots, a_n\}$.  Then $(X_\alpha)_{\-k}$ is $\A^1$-connected.
\end{introtheorem}

Since $\A^1$-connectedness is equivalent to universal $R$-triviality for smooth proper varieties, Theorem \ref{intro thm: norm varieties} says that the standard norm varieties over a field of characteristic $0$ are universally $R$-trivial after base change to the algebraic closure of the base field.  A result of Karpenko and Merkurjev \cite{Karpenko-Merkurjev} states that the standard norm variety $X_\alpha$ is universally $\CH_0$-trivial.  It is known that universal $\CH_0$-triviality is a strictly stronger condition than universal $R$-triviality.  Thus, Theorem \ref{intro thm: norm varieties} gives additional information about $X_\alpha$ on base change to $\-k$.

The proof of Theorem \ref{intro thm: norm varieties} contains two key ingredients, both of which are of independent interest.  The first ingredient is the following criterion for $\A^1$-connectedness of a variety $X$ over a field (see Theorem \ref{theorem:criterion2}). 

\begin{introtheorem}
\label{intro thm: criterion}
Let $X$ be a variety over a perfect field $k$. Suppose that there exists a point $x_0 \in X(k)$ such that for any $x \in X$, one of the following conditions hold:
\begin{enumerate}[label=$(\alph*)$]
\item $\overline{\{x\}}$ contains $x_0$, and $x$ and $x_0$ have the same image in $\pi_0^{\#A^1}(\overline{\{x\}})(k(x))$; 
\item there exists a smooth, irreducible curve in $X$ passing through $x$ and not contained in $\overline{\{x\}}$. 
\end{enumerate}
Then $X$ is $\#A^1$-connected. 
\end{introtheorem}

A straightforward consequence of Theorem \ref{intro thm: criterion} is the characterization of $\A^1$-connectedness of a smooth proper variety $X$ over a field $k$ in terms of the condition that the generic point of $X$ can be connected to a $k$-rational point using naive $\A^1$-homotopies or equivalently, a chain of $\P^1$'s (see Corollary \ref{cor criterion}).  The hypotheses of Theorem \ref{intro thm: criterion} imply that the generic point of $X$ can be connected to $x_0$ by an \emph{$n$-ghost homotopy} in the sense of \cite[Definition 2.7]{Balwe-Sawant-ruled}, thanks to the results of \cite{Balwe-Hogadi-Sawant}.  The proof of Theorem \ref{intro thm: criterion} involves showing that the given $\A^1$-ghost homotopy between the generic point of $X$ and $x_0$ can be systematically enlarged to show that any two points of $X$ over a given essentially smooth field extension $F$ of $k$ map to the same element of $\pi_0^{\A^1}(X)(F)$, which allows one to conclude using \cite[Lemma 6.1.3]{Morel-connectivity}.  Theorem \ref{intro thm: criterion} and several useful consequences of it are proved in Section \ref{section criterion}.

The second key ingredient in the proof is $\A^1$-connectedness of smooth compactifications of the symmetric powers of an $\A^1$-connected variety over a field, studied in Section \ref{section symmetric}.  We show that the symmetric power of an $\A^1$-connected smooth projective variety is $\A^1$-connected (see Theorem \ref{theorem:symmetric}) as a consequence of Theorem \ref{intro thm: criterion} in the appendix.  However, this is not sufficient to conclude $\A^1$-connectedness of smooth compactifications of the symmetric powers as they are not smooth in general.  Since $\A^1$-connectedness is a birationally invariant property of smooth proper schemes, we only need to find one smooth birational proper model of the symmetric power of an $\A^1$-connected smooth projective variety.  This is done with the help of some tricky geometric arguments involving approximating points on rational varieties by rational curves and the geometric structure of symmetric powers.  The main result of Section \ref{section symmetric} is as follows (see Theorem \ref{theorem:symmetric_model}).

\begin{introtheorem}
\label{intro thm: symmetric}
Let $X$ be an $\A^1$-connected smooth projective variety over an algebraically closed field $k$ of characteristic $0$.  Then any smooth proper variety birational to $\Sym^m X$ is $\A^1$-connected, where $m \geq 1$ and $\Sym^m X$ denotes the quotient of $X^m$ by the action of the symmetric group $S_m$ by permutation of factors. 
\end{introtheorem}

Given the above two ingredients, the proof of Theorem \ref{intro thm: norm varieties} follows from the inductive construction of the standard norm varieties and the observation that compactifications of norm hypersurfaces are $\A^1$-connected (see Proposition \ref{prop:norm-hypersurface}).  The details can be found in Section \ref{section norm varieties}.

\subsection*{Acknowledgement}
We thank Aravind Asok, Jean-Louis Colliot-Th\'el\`ene, Bruno Kahn, Burt Totaro, Suraj Yadav and the anonymous referee for their comments. 

\subsection*{Conventions and notation}
By a variety over a field $k$, we mean a reduced and irreducible scheme of finite type over $k$.  We will always work with the big Nisnevich site $Sm/k$ of smooth schemes over a field $k$.  Given a simplicial sheaf of sets $\@X$ on $Sm/k$ and an affine scheme $\Spec A$ over $k$, we will write $\@X(A)$ for $\@X(\Spec A)$ for the sake of brevity.  

For a scheme $X$ (over $\Spec \Z$), the residue field of a scheme-theoretic point $x$ will be denoted by $\kappa(x)$.   In case $X$ is a scheme over over a field $k$, we will sometimes denote the residue field of $x$ by $k(x)$.

\section{Geometric criteria for \texorpdfstring{$\#A^1$}{A1}-connectedness}
\label{section criterion}

We will freely use the notation and terminology from \cite{Morel-Voevodsky}, \cite{Balwe-Hogadi-Sawant} and \cite{Balwe-Sawant-ruled}, especially regarding the sheaves of naive and genuine $\A^1$-connected components. 

For any scheme $U$ over $k$, we let $\sigma_0$ and $\sigma_1$ denote the morphisms $U \to U \times \#A^1$ given by $u \mapsto (u,0)$ and $u \mapsto (u,1)$, respectively.  An \emph{$\#A^1$-homotopy} of $U$ in a scheme $X$ over $k$ is a morphism $h: U \times \A^1 \to X$ and we say that $h$ connects $h(0)$ and $h(1)$.  An \emph{$\#A^1$-chain homotopy} of $U$ in $X$ is a finite sequence $h=(h_1, \ldots, h_r)$ where each $h_i$ is an $\#A^1$-homotopy of $U$ in $X$ such that $h_i(1) = h_{i+1}(0)$ for $1 \leq i \leq r-1$.  In this case, we say that $h_1(0)$ and $h_r(1)$ are \emph{$\#A^1$-chain homotopic}.
 
\begin{definition}
\label{defn movable}
A point $x$ of a scheme $X$ is said to be \emph{movable}, if there exists a smooth, irreducible curve $C$ over $\kappa(x)$, a point $c_0 \in C(\kappa(x))$ and a morphism $\gamma: C \to X$ such that $\gamma(c_0) = x$ and the image of $\gamma$ is not contained in the closure $\overline{\{x\}}$ of $x$ in $X$.
\end{definition}

It is easy to see that every non-generic smooth point of a variety is movable. In fact, a non-movable point on a variety has to be a singular point of every closed subvariety that properly contains it.  The main result of this section is the following geometric criterion for $\A^1$-connectedness.

\begin{theorem}
\label{theorem:criterion2}
Let $X$ be a variety over a perfect field $k$. Suppose that there exists a point $x_0 \in X(k)$ such that for any $x \in X$, one of the following conditions hold:
\begin{enumerate}[label=$(\alph*)$]
\item $\overline{\{x\}}$ contains $x_0$, and the $x$ and $x_0$ have the same image in $\pi_0^{\#A^1}(\overline{\{x\}})(k(x))$. 
\item $x$ is a movable point of $X$. 
\end{enumerate}
Then $X$ is $\#A^1$-connected. 
\end{theorem}

\begin{proof}
Let $L_{\A^1}$ denote the fibrant approximation functor \cite[\textsection 2, Lemma 2.6, page 107]{Morel-Voevodsky} for the $\#A^1$-model structure.  We define $\@X := L_{\#A^1}(X)$, so that $\pi_0^{\#A^1}(X) = \pi_0(\@X)$. 

We will prove that for any point $x$ of $X$, the images of $x$ and $x_0$ in $\pi_0^{\#A^1} (X)(k(x))$ are the same. This will prove that $\pi_0^{\#A^1}(X)(F)$ is a singleton for every finitely generated extension $F/k$. We can then deduce by \cite[Lemma 6.1.3]{Morel-connectivity} that $X$ is $\#A^1$-connected.  The strategy of the proof is to inductively construct a sequence of open subsets 
\[
\phi=:U_{-1} \subset U_0 \subset U_1 \subset \cdots \subset X
\]
such that the following conditions hold:
\begin{itemize}
\item[(A)] For every $i \geq 0$, $Z_i:= U_i\backslash U_{i-1}$ is irreducible and $Z_i$ is non-empty unless $U_{i-1} = X$. 
\item[(B)] For every $i \geq 0$, let $j_i, g_i \in \@X(Z_i)$ be defined as the compositions $Z_i \hookrightarrow X \to \@X$ and $ Z_i \to \Spec k \stackrel{x_0}{\to} X \to \@X$ respectively. Then, there exists a morphism $H_i: Z_i \times \Delta^1 \to \@X$ such that $H_i|_{Z_i \times \{0\}} = j_i$ and $H_i|_{Z_i \times \{1\}} = g_i$. 
\end{itemize}
As $X$ is noetherian, this sequence of open sets is of finite length and must stabilize with $U_N = X$ for some integer $N$. Note that 
$$\pi_0(\@X)(k(X)) = \mathop{\lim}_{U \subset X \text{ open}} \pi_0(\@X)(U).$$ Thus, there exists an open subset $U_0 \subset X$ such that the elements $j_0, g_0$ of $\@X(U_0)$, defined as in condition (B) above, are simplicially homotopic. Thus, we have a morphism of simplicial sheaves $H_0: U_0 \times \Delta^1 \to \@X$ connecting $j_0$ and $g_0$. 

We now suppose that $U_{-1} \subset U_0 \subset \cdots \subset U_i$ have been chosen. Let $Z$ be a component of the closed subset $X \backslash U_i$, let $z$ be its generic point and let $L = k(z)$. We claim that the images of $z$ and $x_0$ in $\@X(L)$ are simplicially homotopic.  One of the two conditions in the statement of the theorem must hold for $z$. If condition (a) holds, our claim is trivially true. Thus, we now assume that (a) does not hold for $z$ and hence, $z$ is movable. 

Thus, there exists a smooth, irreducible curve $C$ over $L$, a point $c_0 \in C(L)$ and a non-constant morphism $\gamma: C \to X$ such that $\gamma(c_0) = z$ and the image of $\gamma$ is not contained in $Z$.  Clearly, $\gamma$ maps the generic point of the curve $C$ into some $Z_{i_0}$ where $i_0 \leq i$.  

Let $V_1 := \gamma^{-1}(U_{i_0}) \cup \{c_0\}$, which is an open subset of $C$.  Choose an \'etale morphism $f_1: V_1 \to \#A^1_L$ such that $f(c_0) = 0$. By replacing $V_1$ by a smaller open neighbourhood of $c_0$ if necessary, we may also assume that $f_1^{-1}(0) = \{c_0\}$. Let $V_2 = \#A^1_L \backslash \{0\}$ and let $f_2: V_2 \to \#A^1_L$ be the inclusion map. Then $\{f_i:V_i \to \#A^1_L|i = 1,2\}$ is an elementary Nisnevich cover of $\#A^1_L$. Let $W:= V_1 \times_{\#A^1_L} V_2$. We observe that the morphism $W \to V_1$ is an open immersion. 

Let $h_1: V_1 \to \@X$ be the morphism 
\[
V_1 \hookrightarrow C \stackrel{\gamma}{\to} X \to \@X
\]
and let $h_2: V_2 \to \@X$ be the morphism
\[
V_2 \to \Spec k \stackrel{x_0}{\to} X \to \@X. 
\]
The morphism $h_1|_W$ is the same as $H_{i_0}|_{U \times \{0\}} \circ \gamma|_W$ and the morphism $h_2|_W$ is the same as $H_{i_0}|_{U \times \{1\}} \circ \gamma|_W$. Thus, $H_{i_0} \circ \gamma|_W$ gives us a simplicial homotopy connecting $h_1|_W$ and $h_2|_W$. 

Thus, we have the diagram
\[
\xymatrix{
\ast \ar[rr] \ar[d] && \@X(V_1) \ar[d] \\
\Delta^1 \ar[rr]_{H_{i_0} \circ \gamma|_W} \ar@{-->}[urr] && \@X(W)
}
\]
of simplicial sets, where the morphism on the right is a fibration since $W \hookrightarrow V_1$ is a cofibration and $\@X$ is simplicially fibrant. Thus, there exists a morphism from $\Delta^1$ to $\@X(V_1)$, indicated by the dotted arrow in the above diagram, making the diagram commutative. We denote this morphism by $H'$. Let $h'_1:= H'|_{V_1 \times \{1\}}$. Then $h'_1$ and $h_2$ can be glued together to define morphism from $\#A^1_L$ to $X$. This shows that the images of $z$ and $x_0$ in $\@X(L)$ are simplicially homotopic, concluding the proof of our claim.  

This simplicial homotopy can be extended to a suitable open subset of $Z$. In other words, there exists an open subset $Z_{i+1}$ of $Z$ such that the following conditions hold:
\begin{itemize}
\item[(1)] The morphisms $j_{i+1}, g_{j+1}: Z_{i+1} \to \@X$, defined as $Z_{i+1} \hookrightarrow X \to \@X$ and $Z_{i+1} \to \Spec k \stackrel{x_0}{\to} X \to \@X$ respectively, are connected by a simplicial homotopy $H_{i+1}: Z_{i+1} \times \Delta^1 \to \@X$. 
\item[(2)] $Z_{i+1}$ does not meet any component of $X \backslash U_i$ other than $Z$.
\end{itemize} 
Condition (2) ensures that $Z_{i+1} \cup U_i$ is an open subset of $X$, which we denote by $U_{i+1}$. This completes the induction step.  
\end{proof}

\begin{remark}
Let $X$ be a variety over $k$ and let $x, y \in X(k)$ be distinct points such that $x$ and $y$ map to the same element of $\pi_0^{\#A^1}(X)(k)$. Then $x$ must be a movable point of $X$. Indeed, let $\-X$ be a compactification of $X$. Then $\pi_0^{\#A^1}(\-X)(k) = \@S(\-X)(k)$, by \cite[Theorem 2.4.3]{Asok-Morel} or \cite[Theorem 2]{Balwe-Hogadi-Sawant}. Thus, $x$ and $y$ are $\#A^1$-chain connected on $\-X$. It follows from this that $x$ is a movable point of $X$. 

This observation implies that if $X/k$ is $\#A^1$-connected and has more than one $k$-rational point, then every $k$-rational point must be movable. This allows us to see that Theorem \ref{theorem:criterion2} may not hold for a variety $X$ if some points of $X$ are not movable. A simple example is given by the curve $C$ in $\#P^2_{\#R}$ defined by the homogeneous polynomial $Y^2 - X^2(X-1)$. This is a rational curve, and so its generic point can certainly be connected to an $\#R$-valued point. However, it is easy to see that the point $(0:0:1)$ is not movable. 
\end{remark}

A useful consequence of Theorem \ref{theorem:criterion2} for smooth proper varieties is that $\A^1$-connectedness can be characterized as the condition that the generic point of the variety can be connected to a $k$-rational point of it by a chain of $\P^1$'s (also see \cite[Theorem 8.5.1]{Kahn-Sujatha} for an equivalent treatment in terms of $R$-equivalence classes).

\begin{corollary}
\label{cor criterion}
A smooth proper variety $X$ over a field $k$ is $\A^1$-connected if and only if $X$ has a $k$-rational point and the generic point of $X$ can be connected to it (and hence, to any $k$-rational point of $X$) by an $\A^1$-chain homotopy.
\end{corollary}
\begin{proof}
By Theorem \ref{theorem:criterion2}, it follows that any smooth variety $X$ is $\A^1$-connected if and only if the generic point of $X$ has the same image in $\pi_0^{\A^1}(X)(k(X))$ as a $k$-rational point of $X$.  If $X$ is proper in addition, the latter statement is equivalent to the statement that the generic point of $X$ can be connected to a $k$-rational point of $X$ by an $\A^1$-chain homotopy, by \cite[Theorem 2.4.3]{Asok-Morel} or \cite[Theorem 2]{Balwe-Hogadi-Sawant}.
\end{proof}

We end this section with two direct consequences of Corollary \ref{cor criterion}, which will be used to show $\A^1$-connectedness of norm varieties in Section \ref{section norm varieties}. 

\begin{corollary}
\label{cor:fiberbase}
Let $k$ be a field of characteristic $0$ and $Y\xrightarrow{f} X$ be a generically smooth morphism of smooth proper varieties over $k$.  If $X$ and the generic fiber of $f$ are $\A^1$-connected, then $Y$ is $\A^1$-connected. 
\end{corollary}
\begin{proof}
Let $\eta$ be the generic point of $X$ and $Y_{\eta}$ be the generic fiber of $f$. Let $\xi$ be the generic point of $Y_{\eta}$. Note that this is also the generic point of $Y$. 

Since $Y_{\eta}$ is $\A^1$-connected, it has a $k(\eta)$-rational point. In other words, the map $f$ has a rational section $X\stackrel{s}{\dashrightarrow}Y$. Since $k$ has characteristic $0$, we may resolve the singularities of the section and find a smooth proper variety $\tilde{X}$ over $k$ and a birational proper morphism $\tilde{X}\xrightarrow{\pi} X$ such that the rational section $s$ defines a morphism $\tilde{X}\xrightarrow{\~s} Y$. 
\[
\xymatrix{
\tilde{X} \ar[rd]_\pi \ar[r]^{\~s} & Y \ar[d]^-f \\
& X\ar@{-->}@/_2pc/[u]_s
}
\]
Since $\tilde{X}$ is birational to $X$ and both $X, \~X$ are smooth, we see that $\A^1$-connectedness of $X$ implies that of $\tilde{X}$ by \cite[Theorem 3]{Asok-crelle}. Thus, its generic point $\eta$ is $\A^1$-chain homotopic to a $k$-rational point $x_0\in X(k)$, which can be chosen to lie in the open set over which $\pi$ is an isomorphism. Note that $\~s(\eta)$ is a $k(\eta)$-rational point of $Y_{\eta}$. Thus by applying $\~s$, we see that $\~s(\eta)$ is $\A^1$-chain homotopic to the $k$-rational point $\~s(x_0)$ in $Y$. However since $Y_\eta$ is an $\A^1$-connected variety over $k(\eta)$, we see that $\~s(\eta)$ is $\A^1$-chain homotopic to its generic point $\xi$, which is also the generic point of $Y$. Thus, the generic point of $Y$ is $\A^1$-chain homotopic to the $k$-rational point $\~s(x_0)$. 
\end{proof}

\begin{remark}
Since $\mathbb A^1$-connectedness is equivalent to universal $R$-triviality, Corollary \ref{cor:fiberbase} is, in a sense, well-known to experts.  Indeed, birational invariance of $R$-equivalence classes in a smooth proper variety has been already observed in \cite[Proposition 10]{Colliot-Thelene-Sansuc-1979} (in characteristic $0$) and in \cite[Corollary 6.6.6]{Kahn-Sujatha} (in arbitrary characteristic).
\end{remark}

\begin{corollary}\label{cor:open}
Let $k$ be a field and $U$ be a variety over $k$, which has a $k$-rational point $x_0\in U(k)$. Assume that for any field extension $F/k$ and any point $x \in U(F)$, there exists a rational map $h:\#P^1_F \dashrightarrow U$, defined at $0,1\in \#P^1_k$ such that
$h(0)=x$ and $h(1)=x_0$. Then any smooth proper model of $U$ is $\A^1$-connected.  
\end{corollary}
\begin{proof}
Let $X$ be a smooth proper model of $U$ and $F=k(X)$ be its function field. Let $\eta \in U(F)$ be the generic point.  By hypothesis, we get a morphism $\#P^1\dashrightarrow X$ connecting $x$ to $x_0$. Since $X$ is proper, this map extends to give a morphism, which in turn gives an $\A^1$-homotopy connecting $x$ to $x_0$.  Theorem \ref{cor criterion} now implies that $X$ is $\A^1$-connected.
\end{proof}

\section{Compactifications of symmetric products of \texorpdfstring{$\A^1$}{A1}-connected varieties}
\label{section symmetric}

The aim of this section is to show that any smooth proper compactification of the symmetric power of an $\A^1$-connected smooth projective variety is $\A^1$-connected.  The proof of this result is quite tricky and relies on two geometric inputs.  The first one is that any point on a rational smooth proper variety can be moved by a rational curve into an open subset that is isomorphic to the open subset of an affine space.  The proof of this result, which uses jet schemes, is discussed in Section \ref{subsection approximation}.  The second input, discussed in Section \ref{subsection symmetric powers}, is the construction of a particular $\A^1$-connected smooth proper compactification of the symmetric power of an $\A^1$-connected smooth projective variety.  The $\A^1$-connectedness of any smooth proper compactification of the symmetric power of an $\A^1$-connected smooth projective variety then follows from birational invariance of $\A^1$-connectedness for smooth proper varieties.  

Throughout this section, we will assume that $k$ is an algebraically closed field of characteristic $0$.

\subsection{Rational approximation of curves}
\label{subsection approximation}

Let $t$ be a variable and for any $n \geq 0$, set $\#D_n:= \Spec k[[t]]/\<t^{n+1}\>$.  Let $\#D_{\infty}$ denote the scheme $\Spec k[[t]]$.  For any variety $X$ over $k$ and any non-negative integer $n$, let $J_n(X)$ denote the sheaf $U \mapsto X(U \times \#D_n)$.  It is well-known that (see \cite{Greenberg-schemata1}) for any finite $n$, the sheaf $J_n(X)$ is a scheme and is called the \emph{$m$-jet scheme} of $X$. 

\begin{definition}
\label{defn parameterized pointed curve}
A \emph{parametrized pointed curve} is defined to be a triple $(C, c, \pi)$ where $C$ is a smooth curve, $c$ is a point of $C$ and $\pi$ is a uniformizing parameter of the ring $\@O_{C,c}$. There exists a unique isomorphism $\widehat{\@O}_{C,c} \xrightarrow{\simeq} \kappa(c)[[t]]$ such that $\pi \mapsto t$. Thus, if $(X,x)$ is a variety, a morphism $f: (C,c) \to (X,x)$ determines a $\kappa(c)$-valued point of $J_{\infty}(X)$ (which depends on $\pi$), which we denote by $f_{\infty}$. The image of $f_{\infty}$ in $J_n(X)$ will be denoted by $f_n$ for any non-negative integer $n$. 
\end{definition}

\begin{definition}
\label{defn approximation}
We say that a variety $X$ over $k$ \emph{admits rational approximation of curves} if it has the property that for any point $x$ of $X$, non-negative integer $n$, parametrized pointed curve $(C, c, \pi)$ and morphism $f: (C,c) \to (X,x)$, there exists a morphism $g: (V,0) \to (X,x)$ for some open subset $U$ of $\A^1$ containing $0$ such that $f_n = g_n$ (here we use the parameter $T$ on $V \subset \#A^1 = \Spec k[T]$).  We say that $g$ approximates $f$ to order $n$.  We will abuse the notation and call such a $g$ a \emph{pointed rational curve} $(\#A^1, 0) \to (X,x)$.
\end{definition}

It is easy to see that the property of admitting rational approximation of curves is respected by blow-downs.  Lemma \ref{lemma rational approximation} below shows that this property is also respected under blowups at smooth centers.
 
\begin{lemma}
\label{lemma rational approximation}
Let $X$ be a smooth variety which admits rational approximation of curves.  If $\~X$ is a blowup of $X$ at a smooth center $Z$, then $\~X$ too admits rational approximation of curves. 
\end{lemma}
\begin{proof}
We fix a point $x$ of $X$. Let $n$ be a non-negative integer. We prove that for any parametrized pointed curve $(C,c, \pi)$ and morphism $f: (C,c) \to (X,x)$, there exists an integer $m$ such that if $g:(\#A^1, 0) \to (X,x)$ approximates $f$ to order $n + m$, then the lift of $g$ to $\~X$ approximates the lift of $f$ to order $n$. Clearly, this will imply the result. Without any loss of generality, we will assume that $\kappa(c) = \kappa(x) = k$. 

If $x \notin Z$, the result is trivial. Thus, we now assume that $x$ is in $Z$. Choose a system of parameters $\pi_1, \ldots, \pi_d$ at $x$ (where $d = \dim(X)$). By replacing $X$ with an open neighbourhood if necessary, we may assume that the $\pi_i$ are regular on $X$ and that $Z$ is defined by the ideal $\<\pi_1, \ldots, \pi_p\>$ for some $p \leq d$. Thus, $\~X$ is the closed subscheme of $X \times \#P^{p-1}$ defined by the equations $T_i \pi_j = T_j \pi_i$ for $1 \leq i,j \leq p$, where $T_1, \ldots, T_p$ denote homogeneous coordinates on $\#P^{p-1}$. 

For $1 \leq i \leq d$, let $\alpha_i = f_{\infty}^*(\pi_i)$.  Let $i_0$ be such that $\ord_t(\alpha_{i_0}) \leq \ord_t (\alpha_i)$ for $1 \leq i \leq p$.  Without loss of generality, we may take $i_0 = 1$.  We set $m:= \ord_{t}(\alpha_1)$.  Since $X$ admits rational approximation of curves, there exists a pointed rational curve $g: (\#A^1,0) \to (X,x)$, which approximates $f$ to order $n + m$.  For $1 \leq i \leq d$, let $\beta_i = g_{\infty}^*(\pi_i)$. Since $g$ approximates $f$ to order $n + m$, we have $\alpha_i \equiv \beta_i \mod(t^{m + n + 1})$ for all $i$. Thus, $\ord_t(\beta_1) = m \leq \ord_t(\beta_i)$ for $1 \leq i \leq p$.  

Let $\~f$ and $\~g$ denote the lifts of $f$ and $g$ to $\~X$. Then $\~f$ maps the closed point of $\#D$ to a point $\~x$ in the open subset of $\~X$ defined by $T_1 \neq 0$.  Let us denote the generic point of $\#D$ by $\eta_0$ and the closed point by $\eta_1$.  Then 
\[
\~f(\eta_0) = (f(\eta_0), (\alpha_1:  \cdots : \alpha_p))  = (f(\eta_0), (\gamma_1:  \cdots : \gamma_p))\in (X \times \#P^{p-1})(k((t))) \text{,}
\]
where $\gamma_i = \alpha_i/\alpha_1 \in k[[t]] \subset k((t))$.  We also have
\[
\~f(\eta_1) = (f(\eta_1), (\gamma_1(0):  \cdots : \gamma_p(0)) \in (X \times \#P^{p-1})(k),
\]
where $\gamma_i(0)$ is the image of $\gamma_i$ under the quotient homomorphism $k[[t]] \to k[[t]]/\<t\> = k$. As $\gamma_1 = 1$, the point $\~f(\eta_1)$ lies in the open subscheme of $\~X$ defined by the condition $T_1 \neq 0$.  Similarly, we have 
\[
\~g(\eta_0) = (g(\eta_0), (\beta_1:  \cdots : \beta_p)) (g(\eta_0), (\delta_1:  \cdots : \delta_p))\in (X \times \#P^{p-1})(k((t))),
\]
where $\delta_i = \beta_i/\beta_1 \in k[[t]]$. Also, 
\[
\~g(\eta_1) = (g(\eta_1), (\delta_1(0):  \cdots : \delta_p(0)) \in (X \times \#P^{p-1})(k).
\]
Again, we observe that $\delta_{1} = 1$ and so the point $\~g(\eta_1)$ lies in the open subcheme $T_1 \neq 0$. 

In fact, since $\gamma_i \equiv \delta_i \mod(t^{n+1})$ for all $i$, we have $\~f(\eta_1) = \~g(\eta_1)$.  At the point $\~f(\eta_1)$, the functions 
\[
\pi_1, T_2/T_1, \ldots, T_p/T_1, \pi_{p+1}, \ldots , \pi_d
\]
form a system of parameters.  Since $\alpha_1  \equiv  \beta_1  \mod(t^{m+n+1})$, $\gamma_i  \equiv  \delta_i \mod(t^{n+1})$ for $2 \leq i \leq p$ and $\alpha_j  \equiv  \delta_j \mod(t^{m+n+1})$ for $p \leq j \leq d$, we conclude that $\~g$ approximates $\~f$ to order $n$. 
\end{proof}

\begin{proposition}
\label{prop rational approximation}
Any smooth rational variety admits rational approximation of curves. 
\end{proposition}

\begin{proof}
It is clear that the projective space admits rational approximation of curves.  Now, the result follows from Lemma \ref{lemma rational approximation} and the weak factorization theorem \cite[Theorem 0.1.1]{AKMW}. 
\end{proof}

\begin{corollary}
Let $X$ be a smooth projective variety and $U \subset X$ be an open subset. Then, for any point $x \in X \backslash U$, there exists a rational curve through $x$ which intersects $U$. 
\end{corollary}

\begin{proof}
Let $Z = X \backslash U$. Choose a pointed pointed parametrized curve $(C,c,\pi)$ with $\kappa(c) = \kappa(x)$, and a morphism $f: (C,c) \to (X,x)$ such that $f(C)$ intersects $U$. If we have $f_n \in J_n(Z)(\kappa(x))$ for all $n \geq 0$, then we must have $f_{\infty} \in J_{\infty}(Z)$. Thus, $f$ must map the generic point of $C$ into $Z$, and hence $f(C) \subset Z$, which is a contradiction. Thus, there exists some  positive integer $n$ such that $f_n \notin J_n(Z)(\kappa(x))$. Since $X$ admits rational approximation of curves, there exists a pointed rational curve $g: (\#A^1,0) \to (X,x)$ which approximates $f$ to order $n$. Then, $g(\#A^1)$ must also intersect $U$. 
\end{proof}

\subsection{\texorpdfstring{$\A^1$}{A1}-connectedness of symmetric powers}
\label{subsection symmetric powers}

\begin{notation}\label{notation}
For an integer $m\geq 2$, we let $\Sym^m X$ denote the $m$-th symmetric power of $X$, that is, the quotient of $X^m$ by the action of the symmetric group $S_m$.  Note that $\Sym^m X$ is singular in general.  Let $\Delta$ denote the union of all the partial diagonals of $X^m$.  The geometric quotient $U:= (X^m \setminus \Delta)/S_m$ is a smooth scheme.  We let $\eta$, $\xi$ and $\theta$ denote the generic points of $X$, $X^m$ and $\Sym^m X$, respectively. Let $K:=k(\eta)$ and $E:=k(\xi)$ and $L:=k(\theta)$.  Since $X$ is geometrically integral, it follows from the Grothendieck-Sharp theorem (see  \cite[Remarque (4.2.1.4)]{EGA4-IV}, \cite[Theorem]{Sharp}) that $\Spec K^{\otimes m}$ is an integral domain and its function field is isomorphic to $E$.  

Let $F$ be any field extension of $k$ in which $k$ is algebraically closed.  Let $F'$ denote the function field of $\Spec F^{\otimes m}$.  For any $F$-rational point $\Spec F \xrightarrow{p} X$, we denote by $p^m$ the corresponding $\Spec F'$-valued point of $X^m$. 
\end{notation}

Given an $\A^1$-connected smooth projective variety $X$ over $k$, we first construct an explicit smooth proper compactification of $\Sym^m X$, which is $\A^1$-connected.  This compactification will be constructed by first considering a particular blowup $\~X$ of $X$ and then taking a suitable compactification of $\Sym^m \~X$.  Lemma \ref{lemma pre-arrange-x} below gives the required $\~X$.

\begin{lemma}
\label{lemma pre-arrange-x} 
Let $X$ be a smooth projective $\A^1$-connected variety over an algebraically closed field $k$ of characteristic $0$.  Let $x_0\in X(k)$. Let $M$ be a finitely generated field extension of $k$ and let $C$ be a smooth curve over $M$ with a point $0\in C(M)$. Let $f: C \to X$ be a morphism with $f(0)=x_0$.  Further, assume that $k$ is algebraically closed in $M$ and that the image of generic point of $C$ in $X$ is a point of dimension $\geq 2$.  Then there exists a proper birational morphism $X'\xrightarrow{\alpha} X$ such that 
\begin{enumerate}[label=$(\arabic*)$]
\item $\alpha$ is obtained by successively blowing up $k$-rational points lying over $x_0$;
\item $x'_0:=f'(0)$ is a $k$-rational point, where $f'$ is the lift of $f$ to $X'$; and
\item if $\tilde{X}\xrightarrow{\pi} X'$ is the blowup of $X'$ at $x'_0$, then $\tilde{f}(0)$ is a positive-dimensional point of $\tilde{X}$, where $\tilde{f}$ is the lift of $f'$ to $\tilde{X}$. 
\end{enumerate}
\end{lemma}
\begin{proof}
Let $f_M:C\to X_M$ be the induced map obtained by the base change of $X$ to $M$.  By replacing $C$ with the normalization of the image of $f_M$, we may assume without loss of generality that $C$ is the normalization of its image.

We first claim that the image of $f_M$ cannot be normal at $f_M(0)$.  To see this, we first \emph{spread} the morphism $f$, that is, consider the diagram 
$$\xymatrix{
C \ar[r]\ar[d]\ar@/^2pc/[rr]^f & \@C \ar[r]_-{\hat{f}}\ar[d] & X \\
\Spec M \ar@/^1.5pc/[u]^0 \ar[r]^-\eta & B\ar@/^1.5pc/[u]^s &  
}$$
in which 
\begin{itemize}
\item $B$ is an integral $k$-scheme and $\eta:\Spec M \to B$ is the generic point of $B$; 
\item $\@C \to B$ is smooth curve with a section $s$; and 
\item the square is cartesian, so the curve $(C,0)$ is the generic fiber of $\@C\to B$. 
\end{itemize}
Since $f(0)$ is a $k$-rational point, $\hat{f}\circ s$ is a constant map.  If the image $f_M(C)$ is normal at $f_M(0)$, then since $C$ also happens to be the normalization of its image, we deduce that $f_M$ is an immersion in a neighbourhood of $0$.  By a standard limiting argument, we can conclude that there exists an open set $U\subset B$ and an open neighbourhood $\@C^0 \subset \@C_U$ of the image of the section $s(U)$, such that $(\hat{f}|_{\@C^0})_M: \@C^0 \to X_{M}$ is an immersion.  However, this is impossible if $\hat{f}\circ s$ is a constant map, since $\dim(s(U)) = \dim(B)\geq 2$ by assumption.  This shows that if $f_M(C)$ is normal at the image of $0$, then $f(0)$ cannot be a $k$-rational point.

Let $\tilde{X}\to X$ be the blowup of $X$ at $f(0)$.  The map $f$ lifts uniquely to give a map $\tilde{f}:C \to \tilde{X}$.  If $\tilde{f}(0)$ is a positive dimensional point, then we are done. Assume the contrary; then $\tilde{f}(0)$ has to be a $k$-rational point, since $k$ is algebraically closed in $M$.  Note that the map $\tilde{X}_M \to X_M$ is also the blowup of $X_M$ at $f_M(0)$.  We may replace $X$ with $\tilde{X}$ without loss of generality. However, by the embedded desingularization of a curve, this process can only be repeated finitely many times until $\tilde{f}_M(C)$ becomes regular at $\tilde{f}_M(0)$ in which case, $\tilde{f}(0)$ cannot be a $k$-rational point as observed in the above paragraph. This completes the proof of the lemma. 

\end{proof}

\begin{lemma}\label{be_positive_be_free}
We follow the notation described in Notation \ref{notation}.  Let $K$ be the function field of a variety $X$ over $k$ and let $p\in X(K)$ be a $K$-rational point, which is not in the image of $X(k)$. Then the image of $p^m$ is not contained in $\Delta$.  
\end{lemma}
\begin{proof}
Let $Z$ be the closure of $p$ (that is, the image of $p$).  Since $K$ is the function field of a $X$, which is a geometrically irreducible variety, we see that $Z$ is geometrically irreducible. In particular, it is zero dimensional if and only if it is a $k$-rational point. However, by hypothesis, $p$ is not $k$-rational. Hence $\dim(Z)>0$. The closure of the image of $p^m$, that is, the closure of the image of 
$$\Spec(E) \xrightarrow{\simeq} \Spec(K)^m \xrightarrow{p^m} X^m$$ 
is precisely $Z^m$. Since $\dim(Z)>0$, it follows that $Z^m$ is not contained in any partial diagonal in $X^m$. 
\end{proof}

\begin{lemma}\label{lemma singlep1} 
Let $X, K, E, L$ be as in Notation \ref{notation}.  Let $p,q \in X(E)$. Let $\phi: \P^1_K \to X$ be a morphism with $\phi(0) = p$ and $\phi(1)=q$.  Then there exists a morphism $$\psi:\P^1_L \to S^mX$$ connecting the images of $p^m$ and $q^m$ in $S^m(X)$. 
\end{lemma}
\begin{proof}
Note that $E$ is the function field of the integral scheme $\Spec K^{\otimes m}$. Note that $S_m$ acts on $\Spec K^{\otimes m}$ by permutation of factors and hence we have an induced $S_m$-action on $E$. We have an $S_m$-equivariant isomorphism 
\[ 
(\P^1_K)^m \xrightarrow{\simeq} (\P^1_k)^m \times_k E,
\]
where the $S_m$-action on the right is the diagonal action. Taking $m$-fold product of the morphism $\phi$, we get an $S_m$-equivariant morphism 
\[
\phi^m:  (\P^1_k)^m \times_k E \cong (\P^1_K)^m \to X^m. 
\]
Now precomposing the above map with the diagonal $\P^1_k \times_k E \xrightarrow {\Delta \times {\rm id}_{E}} (\P^1_k)^m \times_k E$, we get an $S_m$-equivariant morphism
\[
\P^1_k \times_k E \to X^m.
\]
We obtain the desired morphism $\psi$ after taking quotient by $S_m$.
\end{proof}

We are now set to prove the main result of this section.

\begin{theorem}
\label{theorem:symmetric_model} 
Let $X$ be a smooth projective $\A^1$-connected variety over an algebraically closed field $k$ of characteristic $0$ and fix a positive integer $m$.  Then any smooth proper variety birational to $\Sym^m X$ is $\A^1$-connected. 
\end{theorem}

\begin{proof}
Since $\A^1$-connectedness is a birationally invariant property of smooth proper varieties \cite[Theorem 3.9]{Asok-crelle}, it suffices to obtain one smooth proper $\A^1$-connected variety birational to $\Sym^m X$.

Since $X$ is $\A^1$-connected, there exists an $\A^1$-chain homotopy connecting a $k$-rational point $x_0\in X(k)$ to the generic point $\eta$ of $X$ by Corollary \ref{cor criterion}.  Since $X$ is proper, this $\A^1$-chain homotopy gives rise to a morphism 
\[
f :C = \Union^r_{i=1} \P^1_L \xrightarrow{\cup f_i} X
\]
satisfying $f_1(0)=x_0$, $f_i(1) = f_{i+1}(0)$ for $1 \leq i \leq r-1$, and $f_r(1)=\eta$.  Applying Lemma \ref{lemma pre-arrange-x}, we can find a blowup $\pi:\tilde{X}\to X$ (obtained by successively blowing up points over $x_0$) and $\tilde{f_i}$.  If we denote the lift of $f_i$ to $\tilde{X}$ by $\tilde{f}_i$ for every $i$, we see by Lemma \ref{lemma pre-arrange-x} that $\tilde{x}_0:=\tilde{f_1}(0)$ is a positive dimensional point of $\~X$.  Since $\pi:\tilde{X}\to X$ is an isomorphism outside $x_0$, the lifts $\tilde{f_i}$ glue to give a well-defined map $\tilde{f}: C \to \tilde{X}$ lifting $f$.  Thus, $\tilde{f}$ is a chain of rational curves in $\~X$ connecting $\tilde{x}_0$ to $\eta$.  Note that $\tilde{x}_0$ lies in the exceptional divisor $E$ of $\pi$. 

We now apply Lemma \ref{lemma singlep1} to each of the $\tilde{f_i}$ and get a chain of rational curves defined over the generic point $\theta$ of $\Sym^m\tilde{X}$. This chain of curves joins $\tilde{x}_0^m$ to $\theta$.  All nodes of this chain are in the free locus  of $\Sym^m X$, thanks to Lemma \ref{be_positive_be_free}. The point $\tilde{x}_0^m$ is also in the free locus by the same argument. In addition, $\tilde{x}_0^m \in \Sym^m E$ (considered as a subscheme of $\Sym^m \tilde{X}$).  

Note that $E$ is a projective space and consequently, $\Sym^m E$ is rational by \cite{Mattuck} (see also \cite[Chapter 4, Theorem 2.8]{GKZ}).  Let $U$ be an open subscheme of $\Sym^m E$ that is isomorphic to an open subscheme of an affine space.  Therefore, any point of $U(F)$ can be connected by an $\A^1$-homotopy to a point of $U(F)$ lying in the image of the natural map $U(k) \to U(F)$.  If $\tilde{x}_0^m \in U$, then we are done.  If not, we apply Proposition \ref{prop rational approximation} to obtain an $\A^1$-homotopy $h: \A^1 \to \Sym^m E$ connecting $\tilde{x}_0^m$ to a point in $U$.  We thus get a chain of rational curves in $\Sym^m \tilde{X}$, which joins a $k$-rational point in the free locus of $\Sym^m \tilde{X}$ to the generic point of $\Sym^m \tilde{X}$, and all whose nodes are in the free locus. 
 
Let $Y\to \Sym^m \tilde{X}$ be a desingularization which is an isomorphism over the smooth locus of $\Sym^m \tilde{X}$.  Such a desingularization exists since $k$ is assumed to be of characteristic $0$.  Then the chain of rational curves found in the above paragraph lifts to give a chain of rational curves in $Y$.  This gives an $\A^1$-chain homotopy connecting a $k$-rational point of $Y$ to its generic point.  Hence, $Y$ is $\A^1$-connected by Corollary \ref{cor criterion}.  Since $Y$ is also a desingularization of $\Sym^m X$, this completes the proof of the theorem.
\end{proof}

\section{\texorpdfstring{$\A^1$}{A1}-connectedness of norm varieties}
\label{section norm varieties}

In this section, we use the results of Sections \ref{section criterion} and \ref{section symmetric} to prove Theorem \ref{intro thm: norm varieties} stated in the introduction.  We begin with some preliminaries on norm hypersurfaces.

\begin{definition}
\label{definition norm hypersurface}
Let $F$ be a field and let $K/F$ be a finite field extension of degree $n$. If we pick a basis for the $F$-vector space $K$, the norm function $N_{K/F}: K \to F$ is seen to be given by a polynomial in $n$ variables. Thus, there exists a homomorphism of algebraic groups $N: R_{K/F}(\#G_{m,K}) \to \#G_{m,F}$ such that $N_{K/F}$ describes the behaviour of $N$ on $F$-valued points. For any $a \in F^{\times} = \#G_m(F)$, we denote the scheme $N^{-1}(a)$ by $H_{K,a}$ and call it the \emph{norm hypersurface} defined by the extension $K/F$ and the element $a \in F^{\times}$. 
\end{definition}

Clearly, $H_{K,1}(F) \neq \emptyset$. For any $x \in K^{\times}$, the morphism $m_x: \#G_{m,K} \to \#G_{m,K}$, $y \mapsto xy$ induces a $k$-automorphism of $R_{K/k}(m_x)$ of $R_{K/F}(\#G_{m,K})$. This automorphism maps $H_{K,1}$ isomorphically onto $H_{K,N_{K/F}(x)}$. Note that an element $x$ of $K^{\times}$ corresponds to a $F$-rational point of $R_{K/F}(\#G_{K,m})$. Thus, we see that if $H_{K,a}(F) \neq \emptyset$, then $H_{K,a}$ is isomorphic to $H_{K,1}$.

\begin{proposition}
\label{prop:norm-hypersurface}
Let $\ell$ be a prime number.  Let $K/F$ be a separable extension of a field $F$ and let $L/F$ be its Galois closure. Suppose that $\Gal(L/F) = S_{\ell}$, the symmetric group on $\ell$ letters, and that $\Gal(L/K)$ is the stabilizer of one of the letters (so that $\Gal(L/K) \cong S_{\ell-1}$).  If $a\in F^{\times}$ is such that $H_{K,a}(F) \neq \emptyset$, then any birational smooth proper model $X$ of $H_{K,a}$ is $\A^1$-connected. 
\end{proposition}
\begin{proof}
As we saw above, the condition $H_{K,a}(F) \neq \emptyset$ implies that $H_{K,a} \cong H_{K,1}$. Thus, we may assume that $a = 1$.  By \cite[Theorem 4.1]{Endo}, $H_{K,1}$ is retract rational. Now, it follows from \cite[Theorem 2.3.6]{Asok-Morel} that any smooth proper model of $H_{K,1}$ is $\#A^1$-connected. 
\end{proof}

We are now set to put all the ingredients together and give a proof of Theorem \ref{intro thm: norm varieties} stated in the introduction.  We start with a brief recollection of the construction of norm varieties given in \cite[\textsection 2]{Suslin-Joukhovitski}.

Let $X$ be a smooth projective variety over a field $k$ of characteristic $0$ and let $\ell$ be a prime number.  Let $\Sym^{\ell} X = X^\ell/S_\ell$ be the $\ell$th symmetric power of $X$, which need not be smooth in general.  Let $\Delta$ denote the union of all the partial diagonals of $X^{\ell}$.  The geometric quotient $U:= (X^{\ell} \setminus \Delta)/S_{\ell}$ is a smooth scheme.  There is a finite, surjective morphism $\pi: X \times \Sym^{\ell-1} X \to \Sym^{\ell} X$ of degree $\ell$ and its restriction $\pi^{-1}(U) \to U$ is a finite \'etale morphism of degree $\ell$.  The sheaf $$\@E:=\pi_*(\@O_{X \times \Sym^{\ell-1}X})|_U$$ is a locally free sheaf of $\@O_U$-algebras of rank $\ell$.  Let $\#V(\@E)$ denote the associated vector bundle on $U$.  Since $\@E$ is a locally free sheaf of $\@O_U$-algebras, there is a well-defined norm morphism
\[
N: \@E \to \@O_U,
\]
which can be viewed as a section of the degree $\ell$ component $\Sym^{\ell} \@E^{\vee}$ of the symmetric algebra on the dual of $\@E$.  

Any element $a \in k^{\times}$ can be viewed as a section of $\@O_U$.  Let $Y$ denote the closed subscheme of $\#V(\@E)$ defined by the equation $N = a$. Then, $Y$ is smooth over $U$ and geometrically irreducible, by \cite[Lemma 2.1]{Suslin-Joukhovitski}.  Define $N(X, a, \ell)$ to be a smooth compactification of $Y$, which exists since $k$ has characteristic $0$.  We may further arrange that $N(X, a, \ell)$ is proper and surjective over a smooth compactification of $U$ with smooth generic fiber, by resolution of indeterminacy and resolution of singularities.  In other words, $N(X, a, \ell)$ admits a surjective proper morphism to a smooth proper variety birational to $\Sym^{\ell} X$.  This observation will be crucial to our proof of Theorem \ref{intro thm: norm varieties} (see Theorem \ref{theorem:norm varieties} below).  

\begin{definition}
\label{def:norm variety}
Let $n \geq 2$ be an integer.  Let $a_1, \ldots, a_n \in k^{\times}$ and consider the symbol $\alpha := \{a_1, \ldots, a_n\} \in \KM_n(k)/\ell$.  The \emph{norm variety} $X_\alpha$ associated with the symbol $\alpha$ is inductively defined as follows.  Define $X_{\{a_1, a_2\}}$ to be the Severi-Brauer variety associated with the cyclic algebra corresponding to $\{a_1, a_2\}$ (see \cite[Chapter 5]{Gille-Szamuely}).  Inductively proceeding, we define 
\[
X_\alpha = X_{\{a_1, \ldots, a_n\}} := N(X_{\{a_1, \ldots, a_{n-1}\}}, a_n, \ell),
\]
for $n>2$.
\end{definition}

\begin{theorem}
\label{theorem:norm varieties}
Let $k$ be an algebraically closed field of characteristic $0$.  Let $n\geq 2$ be an integer and $a_1, \ldots, a_n \in k^{\times}$.  Then the norm variety $X_\alpha$ associated with the ordered sequence $\alpha = \{a_1, \ldots, a_n\}$ is $\A^1$-connected.
\end{theorem}
\begin{proof}
We proceed by induction on $n$.  As noted in the introduction, in the case $n=2$, the norm variety $X_\alpha$ is the Severi-Brauer variety associated with the cyclic algebra corresponding to the symbol $\{a_1, a_2\} \in \KM_2(k)/\ell$.  In this case, $X_\alpha$ is clearly $\A^1$-connected as it is isomorphic to a projective space.

Now, let $n>2$ and assume by induction that $X = X_{\{a_1, \ldots, a_{n-1}\}}$ is $\A^1$-connected.  Recall from the construction of $X_\alpha$ that it admits a surjective proper morphism to a smooth proper variety birational to $\Sym^{\ell} X$.  The generic fiber of this morphism is a smooth compactification of the norm hypersurface $N=a_n$ over $k(\Sym^\ell X)$.  Since $k$ is algebraically closed and $a_n \in k^\times$, this generic fiber admits a $k(\Sym^\ell X)$-rational point.  We can apply the criterion given by Corollary \ref{cor:fiberbase} to this morphism to conclude that $X_\alpha$ is $\A^1$-connected.  Note that the hypotheses of Corollary \ref{cor:fiberbase} are satisfied by Proposition \ref{prop:norm-hypersurface} and Theorem \ref{theorem:symmetric_model}.
\end{proof}

\appendix

\section{\texorpdfstring{$\A^1$}{A1}-connectedness of symmetric powers}
\label{appendix symmetric}

In this appendix, we show that the symmetric powers of smooth projective $\A^1$-connected varieties over an arbitrary field are $\A^1$-connected.

\begin{theorem}
\label{theorem:symmetric}
Let $X$ be an $\A^1$-connected smooth projective variety over an arbitrary field $k$. Then $\Sym^m X$ is $\A^1$-connected, for any positive integer $m$. 
\end{theorem}
\begin{proof}
Let $[m] = \{1, \ldots, m\}$ and let $A = (A_1, \ldots, A_m)$ be an $m$-tuple of subsets of $[m]$, which partition $[m]$, where some of the $A_i$ may be equal to $\emptyset$. For any field extension $L/k$, a point $x = (x_1, \ldots, x_m)$ of $X^m(L)$ is said to be of \emph{type $A$} if for any $1 \leq l \leq m$ and any $i,j \in A_l$, we have $x_i = x_j$.  We denote the closure of the set of all points of $X^m$ of type $A$ by $(X^m)_A$. Clearly, $$(X^m)_A \cong \prod_{A_i \neq \emptyset} X.$$ 

Let $\mult(A) = (\lambda_1, \ldots, \lambda_m)$, where $\lambda_i$ is the number of indices $j$, $1 \leq j \leq m$ such that $|A_j|= i$.  Let $\pi:X^m \to \Sym^m X$ denote the quotient map. It is clear that the image of $(X^m)_A$ is a closed subvariety depending only on $\mult(A)$ and hence, we denote it by $\Sym^m X_{\mult(A)}$.  Let 
\[
\Lambda_m = \left\{(\lambda_1, \ldots, \lambda_m)~|~ \sum_{i=1}^m \lambda_i = 0, \lambda_i \in \#Z_{\geq 0} \;\forall i \right\} \setminus \{(m,0, \ldots, 0)\}.
\] 
For any $\lambda \in \Lambda_m$, observe that $\sum_{i=1}^m \lambda_i \cdot i$ is a partition of $m$, where the integer $i$ occurs $\lambda_i$ times.  Note that $\lambda$ is not allowed to take the value $(m, 0, \ldots, 0)$. Thus, $\Lambda_m$ describes the partitions of $m$ in which there is at least one integer greater than $1$. Also note that for $\lambda = (\lambda_1, \ldots, \lambda_m) \in \Lambda_m$, we have $\lambda_i < m$ for any $i$. 

Since the action of $S_m$ on $\pi^{-1}(\Sym^m X \setminus \bigcup_{\lambda \in \Lambda_m} \Sym^m X_{\lambda})$ is free, we see that the singular locus of $\Sym^m X$ is contained in $\bigcup_{\lambda \in \Lambda_m} \Sym^m X_{\lambda}$. 

Fix a $\lambda = (\lambda_1, \ldots, \lambda_m) \in \Lambda_m$. We define 
\[
A_i := \left\{j \in [m]~|~ \sum_{l=1}^{i-1} \lambda_l < j \leq \sum_{l=1}^i \lambda_l \right\}.
\]
Let $A = (A_1, \ldots, A_m)$.  Then $\mult(A) = \lambda$. 

Let $H:= \{\sigma \in S_m ~|~ \sigma(x) = x, \text{ for all } x \in (X^m)_{A}\}$. Then $H$ is isomorphic to $\prod_{i=1}^m S_{\lambda_i}$. The variety $\Sym^m X_{\lambda}$ is the quotient of $(X^m)_A$ under the action of $H$. Thus, $\Sym^m X_{\lambda}$ is isomorphic to the quotient of  
\[
\prod_{A_i \neq \emptyset} X = \prod_{\lambda_i \neq 0} X^{\lambda_i},
\]
by the action of $\prod_{i=1}^m S_{\lambda_i}$, where $S_{\lambda_i}$ acts on the factor $X^{\lambda_i}$ by permutation of coordinates.  Thus, $\Sym^{m} X_{\lambda}$ is isomorphic to $\prod_{\lambda_i \neq 0}^m \Sym^{\lambda_i} X$. 

We fix a point $x_0 \in X(k)$. By Lemma \ref{lemma singlep1}, the generic point of $\Sym^m X$ is connected to $x_0^m$ by a rational curve. For any integer $m$, let $\@P(m)$ be the statement that for any point $x$ of $\Sym^m X$, one of the following conditions holds:
\begin{itemize}
\item[(1)] $\overline{\{x\}}$ contains $x_0^m$, and $x$ and $x_0^m$ have the same image in $\pi_0^{\#A^1}(\overline{\{x\}})(\kappa(x))$. 
\item[(2)] $x$ is a movable point of $X$ (in the sense of Definition \ref{defn movable}). 
\end{itemize}
We see by Theorem \ref{theorem:criterion2} that proving the statement $\@P(m)$ for all $m \in \N$ will complete the proof. Since any point outside the closed subscheme $\bigcup_{\lambda \in \Lambda_m} \Sym^m X_{\lambda}$ is a smooth point (and hence, movable), we see that it suffices to show that $\Sym^m X_{\lambda}$ is $\#A^1$-connected and contains $x_0^m$ for any $\lambda \in \Lambda_m$. The fact that $x_0^m \in \Sym^m X_{\lambda}$ follows form the definitions, and so it remains to prove that this subscheme is $\#A^1$-connected. 

We prove the statement $\@P(m)$ by induction on $m$. The case $m = 1$ is trivial.  Indeed, as $X$ is smooth, every point of $X$ other than the generic point is movable. Let $x_0$ be any element of $X(k)$.  As $X$ is $\#A^1$-connected, $\eta$ and $x_0$ map to the same element of $\pi_0^{\#A^1}(X)(K)$.  

Suppose that $\@P(l)$ is known to be true for any  $l < m$. Then, for any $\lambda \in \Lambda_m$, we see that $$\Sym^{m} X_{\lambda} = \prod_{\lambda_i \neq 0}^m \Sym^{\lambda_i} X$$ is $\#A^1$-connected (since $\lambda_i < m$ for every $i$) as it is a product of $\#A^1$-connected varieties. This completes the proof. 
\end{proof}


\begin{thebibliography}{9999}

\bibitem{AKMW}
D. Abramovich, K. Karu, K. Matsuki, J. W{\l}odarczyk:
\emph{Torification and factorization of birational maps},
J. Amer. Math. Soc. 15(3) (2002), 531--572.

\bibitem{Asok-crelle}
A. Asok:
\emph{Birational invariants and $\A^1$-connectedness},
J. Reine Angew. Math. 681 (2013), 39--64.

\bibitem{Asok-unramified}
A. Asok:
\emph{Rationality problems and conjectures of Milnor and Bloch-Kato},
Compos. Math. 149 (2013), no. 8, 1312--1326.

\bibitem{Asok-Morel}
A. Asok, F. Morel:
\emph{Smooth varieties up to $\A^1$-homotopy and algebraic $h$-cobordisms},
Adv. Math. 227(2011), no. 5, 1990--2058.

\bibitem{Balwe-Hogadi-Sawant}
C. Balwe, A. Hogadi, A. Sawant:
\emph{$\A^1$-connected components of schemes},
Adv. Math. 282 (2015), 335--361. 

\bibitem{Balwe-Sawant-ruled}
C. Balwe and A. Sawant:
\emph{$\#A^1$-connected components of ruled surfaces},
Preprint, arXiv:1911.05549 [math.AG], Geom. Topol. (2021) to appear.

\bibitem{Colliot-Thelene-Sansuc-1979}  
J.-L. Colliot-Th\'el\`ene and J.-J. Sansuc:
\emph{La R-\'equivalence sur les tores},
Ann. Sci. \'Ecole Norm. Sup. (4) 10 (1977), no. 2, 175--229.

\bibitem{CT-Sansuc-rationality}
J.-L. Colliot-Th\'el\`ene, J.-J. Sansuc:
\emph{The rationality problem for fields of invariants under linear algebraic groups (with special regards to the Brauer group)},
Algebraic groups and homogeneous spaces, 113--186, Tata Inst. Fund. Res. Stud. Math., 19, 2007.

\bibitem{Endo}
S. Endo: 
\emph{The rationality problem for norm one tori},
Nagoya Math. J. 202 (2011), 83--106.

\bibitem{GKZ}
I. Gelfand, M. Kapranov, A. Zelevinsky: 
\emph{Discriminants, resultants and multidimensional
determinants}, Birkh\"auser, Boston, 1994.

\bibitem{Gille-Szamuely}
P. Gille, T. Szamuely:
\emph{Central Simple Algebras and Galois Cohomology}, Cambridge Studies in Advanced Mathematics 101, Cambridge University Press, 2006.

\bibitem{Greenberg-schemata1}
M. Greenberg:
\emph{Schemata over local rings}, 
Ann. of Math. (2) 73 (1961), 624--648.

\bibitem{EGA4-IV}
A. Grothendieck: 
\emph{\'El\'ements de g\'eom\'etrie alg\'ebrique. IV. \'Etude locale des sch\'emas et des morphismes de sch\'emas IV},
Inst. Hautes \'Etudes Sci. Publ. Math. No. 32 (1967).

\bibitem{Kahn-Sujatha}
B. Kahn, R. Sujatha:
\emph{Birational geometry and localisation of categories}. With appendices by Jean-Louis Colliot-Th\'el\`ene and Ofer Gabber, Doc. Math. 2015, Extra vol.: Alexander S. Merkurjev's sixtieth birthday, 277--334.

\bibitem{Karpenko-Merkurjev}
N. Karpenko, A. Merkurjev:
\emph{On standard norm varieties},
Ann. Sci. \'Ec. Norm. Sup{\'e}r. (4) 46 (2013), no. 1, 175--214.

\bibitem{Mattuck}
A. Mattuck:
\emph{The field of multisymmetric functions},
Proc. Amer. Math. Soc. 19 (1968), 764--765.

\bibitem{Morel-connectivity}
F. Morel:
\emph{The stable $\mathbb A^1$-connectivity theorems},
K-Theory 35 (2005), 1--68.

\bibitem{Morel-Voevodsky}
F. Morel, V. Voevodsky:
\emph{$\A^1$-homotopy theory of schemes},
Inst. Hautes \'Etudes Sci. Publ. Math. 90 (1999), 45--143.

\bibitem{Nguyen}
D. H. Nguyen:
\emph{Standard norm varieties for Milnor symbols mod $p$},
Ann. K-Theory 1 (2016), no. 4, 457--475.

\bibitem{Rost-ICM2002}
M. Rost:
\emph{Norm varieties and algebraic cobordism},
Proceedings of the International Congress of Mathematicians, Vol. II (Beijing, 2002), 77--85, Higher Ed. Press, Beijing, 2002. 

\bibitem{Sharp}
R. Y. Sharp:
\emph{The dimension of the tensor product of two field extensions}, 
Bull. London Math. Soc. 9 (1977), no. 1, 42--48. 

\bibitem{Suslin-Joukhovitski}
A. Suslin, S. Joukhovitski:
\emph{Norm varieties},
J. Pure Appl. Algebra 206 (2006), 245--276.

\bibitem{Voevodsky-Bloch-Kato}
V. Voevodsky:
\emph{Motivic cohomology with $\mathbb Z/\ell$-coefficients},
Ann. of Math. 174(2011), 401--438.


\end{thebibliography}
\end{document}